\documentclass[12pt,letterpaper]{article}
\usepackage{amsmath,amsfonts,amsthm,amssymb}

\swapnumbers 

\newtheorem{lemma}{Lemma}[section]

\newtheorem{theorem}[lemma]{Theorem}

\theoremstyle{definition} 
\newtheorem{definition}[lemma]{Definition}
\newtheorem{remark}[lemma]{Remark}

\newcommand\reals{{\mathbb R}}

\begin{document}

\title{Every synaptic algebra has the monotone square
root property}

\author{David J. Foulis{\footnote{Emeritus Professor, Department of
Mathematics and Statistics, University of Massachusetts, Amherst,
MA; Postal Address: 1 Sutton Court, Amherst, MA 01002, USA;
foulis@math.umass.edu.}}\hspace{.05 in}, Anna Jen\v cov\'a  and Sylvia
Pulmannov\'{a}{\footnote{ Mathematical Institute, Slovak Academy of
Sciences, \v Stef\'anikova 49, SK-814 73 Bratislava, Slovakia;
pulmann@mat.savba.sk. The second and third authors were supported by
Research and Development Support Agency under the contract No.
APVV-0178-11 and grant VEGA 2/0059/12.}}}

\date{}

\maketitle

\begin{abstract}
A synaptic algebra is a common generalization of several ordered
algebraic structures based on algebras of self-adjoint operators,
including the self-adjoint part of an AW$\sp{\ast}$-algebra. In
this paper we prove that a synaptic algebra $A$ has the monotone
square root property, i.e., if $0\leq a,b\in A$, then $a\leq b
\Rightarrow a\sp{1/2}\leq b\sp{1/2}$.
\end{abstract}


\medskip


\section{Introduction}

Synaptic algebras, which are generalizations of the self-adjoint part
of an AW$\sp{\ast}$-algebra and of a Rickart C$\sp{\ast}$-algebra,
were introduced in \cite{FSA} and further studied in \cite{FPproj,
TDSyn, SymSyn, ComSyn, 2PSyn, P&ESyn, VLSyn, PuNote}. For additional
examples of synaptic algebras, see the cited references. Synaptic
algebras provide natural representations for notions such as observables
and states featured in the study of the mathematical foundations of
quantum mechanics \cite{S&OSA}.

Each synaptic algebra $A$ is a partially ordered real linear subspace
of a corresponding real linear associative algebra $R$ with unit element
$1$. The algebra $R$ is called the \emph{enveloping algebra} of $A$, and
it is assumed that $1\in A$ and that $A$ is an order-unit normed space
with order unit $1$ \cite[pp. 67--69]{Alf}. The positive cone in $A$ is
denoted by $A\sp{+}=\{a\in A: 0\leq a\}$.

\emph{We assume in what follows that $A$ is a synaptic algebra with
enveloping algebra $R$} \cite[Definition 1.1]{FSA}. To avoid trivialities,
we assume that $1\not=0$, which enables us to identify each $\lambda
\in\reals$ (the ordered field of real numbers) with the element $\lambda1
\in A$.  Also in what follows, the notation `iff' abbreviates `if and
only if' and `:=' means `equals by definition.'

Let $a,b\in A$. Then it is understood that the product $ab$ is calculated
in the enveloping algebra $R$ and that it may or may not belong to $A$.
But if $a$ commutes with $b$, in symbols $aCb$, then $ab=ba\in A$. In
particular, $a\sp{2}\in A$, and $A\sp{+}=\{a\sp{2}:a\in A\}$.

If $a\in A\sp{+}$, there exists a unique $a\sp{1/2}\in A\sp{+}$---the
\emph{square root of $a$}---such that $(a\sp{1/2})\sp{2}=a$. The
\emph{absolute value} of $a$ is denoted and defined by $|a|:=(a\sp{2})
\sp{1/2}$. Clearly, $a\in A\sp{+}$ iff $a=|a|$.

Our purpose in this paper is to prove that $A$ has the following
\emph{monotone square root} (MSR) \emph{property}:
\[
\text{If\ }a,b\in A\sp{+}\text{\ and\ }a\leq b,\text{\ then\ }a\sp{1/2}
\leq b\sp{1/2}.
\]
We note that, by \cite[Proposition 4.2.8]{KadRing}, the self-adjoint part
of a C$\sp{\ast}$-algebra has the MSR property. The MSR property plays
an important role, for instance, in the study of vector lattices in
operator algebras \cite{VLSyn, TopVL}.

\begin{remark} \label{rm:ComMSR}
According to \cite[Lemma 3.3]{VLSyn}, if $a,b\in A\sp{+}$ and $aCb$,
then $a\leq b\Leftrightarrow a\sp{1/2}\leq b\sp{1/2}$. Thus the MSR
property holds in the special case when the elements involved commute.
\end{remark}

\section{Preliminaries}

In this section we attend to some definitions, notation, and facts
that will be needed for our proof that $A$ has the MSR property
(Section \ref{sc:ProofMSR} below).

The order-unit norm of $a\in A$ is denoted and defined by
\[
\|a\|:=\inf\{0<\lambda\in\reals:-\lambda\leq a\leq\lambda\}.
\]
(Recall that $\lambda$ is identified with $\lambda 1$.) In what
follows, limits calculated in $A$ are understood to be limits with
respect to the norm $\|\cdot\|$. We shall need the properties of
the norm as per the following lemma.

\begin{lemma} \label{lm:normprops}
Let $a,b\in A\sp{+}$. Then{\rm: (i)} $a\leq b\Rightarrow
\|a\|\leq\|b\|$. {\rm(ii)} $\|a\sp{1/2}\|=\|a\|\sp{1/2}$.
{\rm(iii)} $A\sp{+}$ is norm closed.
\end{lemma}

\begin{proof}
Part (i) follows from \cite[Proposition 7.12 (c)]{Good}, (ii) is a
consequence of \cite[Lemma 1.7 (ii)]{FSA}, and (iii) follows from
\cite[Theorem 4.7 (iii)]{FSA}.
\end{proof}

If $a,b\in A$, then $aba\in A$ and the \emph{quadratic mapping}
$b\mapsto aba$ is both linear and order preserving on $A$ \cite
[Theorem 4.2]{FSA}.

An idempotent element $p=p\sp{2}\in A$ is called a \emph{projection}
and the set of all projections in $A$ is denoted by $P$. Partially
ordered by the restriction of the partial order on $A$, it turns
out that $P$ is an \emph{orthomodular lattice} (OML) with $p\mapsto p
\sp{\perp}:=1-p$ as the orthocomplementation \cite[\S 5]{FSA}.

An element $a\in A$ is \emph{invertible} iff $a$ has a (necessarily
unique) \emph{inverse} $a\sp{-1}\in A$ such that $aa\sp{-1}=a\sp{-1}a=1$.

\begin{lemma}\label{lm:inverse}
Let $a,b\in A$. Then{\rm:}
\begin{enumerate}
\item $a$ is invertible iff there exists $0<\epsilon\in\reals$ such that
 $\epsilon\leq|a|$.
\item If $0\leq a$ and $a$ is invertible, then  $0\leq a^{-1}$ and
 $a\sp{1/2}$ is invertible.
\item If $0\leq a\leq b$ and $a$ is invertible, then $b$ is invertible
 and $0\leq b\sp{-1} \leq a\sp{-1}$.
\end{enumerate}
\end{lemma}

\begin{proof} (i) holds by \cite[Lemma 7.2]{FSA}. To prove (ii), assume
that $0\leq a$. Then $0\leq a\sp{-1}$ by \cite[Lemma 7.1]{FSA}. Also,
by (i), there exists $0<\epsilon\in\reals$ with $\epsilon\leq a$, and
since $\epsilon Ca$, it follows from Remark \ref{rm:ComMSR} that
$\epsilon\sp{1/2}\leq a\sp{1/2}$, so $a\sp{1/2}$ is invertible.

(iii) In our proof of part (iii), we use properties of quadratic
mappings and we also use the fact that if $e\in A$ and $0\leq e
\leq 1$, then $0\leq e\sp{2}\leq e$ \cite[Lemma 2.5 (i)]{FSA}.
So assume that $0\leq a\leq b$ and $a$ is invertible. Then $b$ and $b
\sp{1/2}$ are invertible by (i) and (ii), whence $0\leq b\sp{-1/2}ab
\sp{-1/2}\leq b\sp{-1/2}bb\sp{-1/2}=1$, and therefore $0\leq(b\sp{-1/2}ab
\sp{-1/2})\sp{2}=b\sp{-1/2}ab\sp{-1}ab\sp{-1/2}\leq b\sp{-1/2}ab\sp{-1/2}$.
Multiplying the latter inequality on both sides, first by $b\sp{1/2}$,
then by $a\sp{-1}$, we obtain $b\sp{-1}\leq a\sp{-1}$.
\end{proof}

Let $a\in A$ and $B\subseteq A$. We define $C(a):=\{b\in A:aCb\}$,
$C(B):=\bigcap\sb{b\in B}C(b)$, and $CC(B):=C(C(B))$. The subset
$B$ of $A$ is said to be \emph{commutative} iff $a,b\in B\Rightarrow
aCb$, i.e., iff $B\subseteq C(B)$. If $B$ is commutative, then so is
$CC(B)$ and $B\subseteq CC(B)$. A \emph{C-block} in $A$ is defined
to be a maximal commutative subset of $A$ \cite[\S 5]{FPRegGHA}.
Evidently, $B$ is a C-block in $A$ iff $B=C(B)$ and by Zorn's lemma,
any commutative subset of $A$, in particular any singleton set $\{a\}$,
can be extended to a C-block.

Suppose that $B\subseteq A$ is a C-block. Then $B$ is closed under
the formation of square roots and inverses, $B$ is a so-called
\emph{sub-synaptic algebra} of $A$ \cite[Definition 2.6]{FPproj},
and $B$ is a commutative synaptic algebra in its own right \cite
[Theorem 2.7]{FPproj}.

If $A$ is a commutative synaptic algebra, then $A$ is a commutative,
associative, partially ordered, Archimedean, real linear algebra
with a unity element $1$ that is an order unit; it is a normed
linear algebra under the order-unit norm; and it may be regarded as
its own enveloping algebra.  By \cite[Theorem 5.11]{VLSyn}, $A$ is
commutative iff $A$ is a vector lattice iff the OML $P$ is a Boolean
algebra. For a commutative synaptic algebra, we have the following
functional representation theorem \cite[Theorem 4.1]{FPproj}.

\begin{theorem} \label{th:FunctionalRep}
Suppose that the synaptic algebra $A$ is commutative, let $X$ be the
Stone space of the Boolean algebra $P$, and denote by $C(X,\reals)$
the partially ordered commutative Banach algebra, with pointwise
operations and partial order and with the supremum {\rm(}or uniform{\rm)}
norm, of all continuous real-valued functions on $X$. Then there is a
subalgebra $F$ of $C(X,\reals)$ such that{\rm:}
\begin{enumerate}
\item The Boolean algebra $P(X,\reals)\subseteq C(X,\reals)$ of all
 characteristic set functions of compact open subsets of $X$ is
 contained in $F$.
\item $F$ is a commutative synaptic algebra with unit $1$ {\rm(}the
 constant function $x\mapsto 1${\rm)} under the operations and
 partial order inherited from $C(X,\reals)$, and the order-unit norm
 on $F$ is the supremum norm.
\item There exists a synaptic isomorphism {\rm(\cite[Definition 2.9]
{FPproj})} $\Psi\colon A\to F$ of $A$ onto $F$ such that the restriction
 of $\Psi$ to $P$ is the Boolean isomorphism of $P$ onto $P(X,\reals)$
corresponding to Stone's representation theorem.
\end{enumerate}
\end{theorem}

Since the synaptic algebra $A$ is an order-unit space, the following
definition \cite[p. 72]{Alf} applies.

\begin{definition}
A \emph{state} on $A$ is a linear functional $\omega\colon A\to\reals$ such
that (1) $\omega$ is positive, i.e., $a\in A\sp{+}\Rightarrow0\leq\omega(a)$
and (2) $\omega(1)=1$. The set of all states on $A$, called the \emph{state
space} of $A$, is denoted by $S(A)$.
\end{definition}

See \cite[Proposition II.1.7]{Alf} and \cite[Corollary II.1.5]{Alf} for a
proof of the next theorem.

\begin{theorem} \label{th:FnlProps}
Let $a\in A$ and let $\rho\colon A\to\reals$ be a nonzero linear functional
on $A$. Then{\rm:}
\begin{enumerate}
\item $a\in A\sp{+}$ iff $0\leq\omega(a)$ for all $\omega\in S(A)$.
\item $\|a\|=\sup\{|\omega(a)|:\omega\in S(A)\}$.
\item $\rho$ is positive iff it is bounded with $\|\rho\|=\rho(1)$.
\item $\rho\in S(A)$ iff $\|\rho\|=\rho(1)=1$.
\end{enumerate}
\end{theorem}

\noindent As a consequence of parts (i) and (ii) of Theorem \ref
{th:FnlProps}, the states on $A$ determine both the partial order
$\leq$ and the norm $\|\cdot\|$ on $A$.

\section{A sufficient condition for the MSR\newline property}

In this section we prove that if the MSR property holds for the special
case in which the elements involved are invertible, then $A$ has the
MSR property (Theorem \ref{th:invertibleMSRcase} below).

\begin{lemma} \label{lm:useful}
If $a\in A\sp{+}$ and $n=1,2,3,...$, then{\rm: (i)} $a+1/n\in A\sp{+}$.
{\rm(ii)} $a+1/n$ is invertible. {\rm(iii)} $\lim\sb{n\rightarrow
\infty}(a+1/n)\sp{1/2}=a\sp{1/2}$.
\end{lemma}

\begin{proof}
Assume that $a\in A\sp{+}$ and that $n$ is a positive integer.
Obviously, (i) holds, and as $1/n\leq a+1/n$, (ii) follows from Lemma
\ref{lm:inverse} (i). Also, by Remark \ref{rm:ComMSR}, $(1/n)\sp{1/2}
\leq(a+1/n)\sp{1/2}\leq(a+1/n)\sp{1/2}+a\sp{1/2}$, hence $(a+1/n)
\sp{1/2}+a\sp{1/2}$ is also invertible. Thus,
\setcounter{equation}{0}
\[
(a+1/n)\sp{1/2}-a\sp{1/2}=[(a+1/n)\sp{1/2}-a\sp{1/2}][(a+1/n)\sp{1/2}
 +a\sp{1/2}][(a+1/n)\sp{1/2}+a\sp{1/2}]\sp{-1}
\]
\begin{equation} \label{eq:001}
=[a+1/n-a][(a+1/n)\sp{1/2}+a\sp{1/2}]\sp{-1}=(1/n)[(a+1/n)\sp{1/2}+a
 \sp{1/2}]\sp{-1}.
\end{equation}
Furthermore, as $a\leq a+1/n$, it follows from Remark \ref{rm:ComMSR}
that, in (\ref{eq:001}), $0\leq(a+1/n)\sp{1/2}-a\sp{1/2}$. Again by
Remark \ref{rm:ComMSR}, we have $(a+1)\sp{1/2}\leq(na+1)\sp{1/2}$, so
\[
0\leq(1/n\sp{1/2})(a+1)\sp{1/2}\leq(1/n)\sp{1/2}(na+1)\sp{1/2}=
[(na+1)/n]\sp{1/2}
\]
\[
=(a+1/n)\sp{1/2}\leq(a+1/n)\sp{1/2}+a\sp{1/2},
\]
whence by Lemma \ref{lm:inverse} (iii),
\[
0\leq[(a+1/n)\sp{1/2}+a\sp{1/2}]\sp{-1}\leq[(1/n\sp{1/2})(a+1)
 \sp{1/2}]\sp{-1}=n\sp{1/2}(a+1)\sp{-1/2},
\]
and therefore
\begin{equation} \label{eq:002}
0\leq(1/n)[(a+1/n)\sp{1/2}+a\sp{1/2}]\sp{-1}\leq(1/n\sp{1/2})
 (a+1)\sp{-1/2}.
\end{equation}
Combining (\ref{eq:001}) and (\ref{eq:002}), we find that
\begin{equation} \label{eq:003}
0\leq (a+1/n)\sp{1/2}-a\sp{1/2}\leq(1/n\sp{1/2})(a+1)\sp{-1/2}.
\end{equation}
By (\ref{eq:003}) and Lemma \ref{lm:normprops} (i), we infer that
\begin{equation} \label{eq:004}
\|(a+1/n)\sp{1/2}-a\sp{1/2}\|\leq(1/n\sp{1/2})\|(a+1)\sp{-1/2}\|,
\end{equation}
from which (iii) follows.
\end{proof}

\begin{theorem} \label{th:invertibleMSRcase}
Suppose that, whenever $a,b\in A\sp{+}$, both $a$ and $b$ are
invertible, and $a\leq b$, then $a\sp{1/2}\leq b\sp{1/2}$. Then
$A$ has the MSR property.
\end{theorem}

\begin{proof}
Assume the hypothesis of the theorem, suppose that $a,b\in A\sp{+}$
with $a\leq b$, and let $n$ be a positive integer. Then $a+1/n,\,
b+1/n\in A\sp{+}$ and both are invertible by Lemma \ref{lm:useful}
(i) and (ii). Clearly, $a+1/n\leq b+1/n$, whence $(b+1/n)\sp{1/2}
-(a+1/n)\sp{1/2}\in A\sp{+}$ by our hypothesis, and by Lemmas
\ref{lm:useful} (iii) and \ref{lm:normprops} (iii) we have
\[
b\sp{1/2}-a\sp{1/2}=\lim\sb{n\rightarrow\infty}\left((b+1/n)\sp{1/2}-
(a+1/n)\sp{1/2}\right)\in A\sp{+},
\]
whereupon $a\sp{1/2}\leq b\sp{1/2}$.
\end{proof}

\section{Proof of the MSR property} \label{sc:ProofMSR}

\begin{lemma} \label{lm:C-block}
Let $B$ be a C-block in $A$, suppose that $\omega\in S(A)$, let $\omega
\sb{0}$ be the restriction of $\omega$ to the commutative synaptic
algebra $B$, and let $X$ be the Stone space of the Boolean algebra $P
\cap B$ of projections in $B$. Then{\rm:}
\begin{enumerate}
\item There exists a subalgebra $F$ of the commutative Banach algebra
 $C(X,\reals)$ such that $F$ is a commutative synaptic algebra, and
there is a synaptic isomorphism $\Psi\colon B\to F$ of $B$ onto $F$.
\item $\omega\sb{0}\circ\Psi\sp{-1}\in S(F)$ and $\omega\sb{0}\circ
 \Psi\sp{-1}$ can be extended to a bounded positive linear functional
 $\hat{\omega}$ on $C(X,\reals)$ with preservation of norm.
\item There is a Borel measure $m\sb{\omega}$ on $X$ such that
\[
\omega(\Psi\sp{-1}(f))=\int\sb{X}f(x)\, dm\sb{\omega}(x)\text
{\ for all\ }f\in F.
\]
\item Suppose that $0\leq a\in B$, $a$ is invertible, $f:=\Psi(a)\in F$,
 and $0\leq\lambda\in\reals$. Then $\lambda+a$ is invertible in $B$,
 $\lambda+f$ is invertible in $F$, and
\[
\omega\left(a(\lambda+a)\sp{-1}\right)=\int\sb{X}\frac{f(x)}{\lambda +f(x)}\,
 dm\sb{\omega}(x).
\]
\end{enumerate}
\end{lemma}

\begin{proof}
(i) Part (i) follows from Theorem \ref{th:FunctionalRep}.

(ii) Clearly $\omega\sb{0}\in S(B)$, and it follows that $\omega\sb{0}
\circ\Psi\sp{-1}\in S(F)$. The existence of an extension of $\omega\sb{0}
\circ\Psi\sp{-1}$ to a bounded linear functional $\hat{\omega}$ with the
same norm on $C(X,\reals)$ follows from the Hahn-Banach extension theorem
\cite[Theorem 1.6.1]{KadRing}. Thus $\hat{\omega}(1)=(\omega\sb{0}\circ
\Psi\sp{-1})(1)=\|\omega\sb{0}\circ\Psi\sp{-1}\|=\|\hat{\omega}\|$,
and by Theorem \ref{th:FnlProps} (iv), $\hat{\omega}$ is positive.

(iii) By the Riesz representation theorem \cite[p. 247, Theorem D]{Halmos},
there is a Borel measure $m\sb{\omega}$ on $X$ such that, for all $f\in
C(X,\reals)$, $\hat{\omega}(f)=\int\sb{X}f(x)dm\sb{\omega}(x)$, from which
$\omega(\Psi\sp{-1}(f))=\int\sb{X}f(x)dm\sb{\omega}(x)$ follows.

(iv) Assume the hypotheses of (iv). As $0\leq a\leq\lambda+a\in B$, it
follows from Lemma \ref{lm:inverse} that $\lambda+a$ is invertible in
$B$, and since $\Psi\colon B\to F$ is a synaptic isomorphism, $\Psi
(\lambda+a)=\lambda+f$ is invertible in $F$. The integral formula in
(iv) then follows upon replacing $f$ in (iii) by $\frac{f}{\lambda+f}$.
\end{proof}

The integral formula for the square root in the proof of the next
theorem is suggested by \cite[(V.5) p.116]{Bh}.

\begin{theorem} \label{th:omegaSRintformula}
There exists a positive $\sigma$-finite Borel measure $\mu$ on
 $(0,\infty)\subseteq\reals$ such that, for every invertible element
 $a\in A\sp{+}$ and every state $\omega\in S(A)$,
\[
\omega(a\sp{1/2})=\int\sb{0}\sp{\infty}\omega\left(a(\lambda+a)
 \sp{-1}\right) d\mu(\lambda).
\]
\end{theorem}

\begin{proof}
Suppose that $0<t\in\reals$. By the substitution $\lambda=tz\sp{2}$
with $0\leq z$, we find that
\[
\int\sb{0}\sp{\infty}\frac{t}{\lambda+t}\lambda\sp{-1/2}d\lambda
 =2t\sp{1/2}\int\sb{0}\sp{\infty}\frac{dz}{1+z\sp{2}}=2t\sp{1/2}
 \frac{\pi}{2}=\pi t\sp{1/2},
\]
whence
\setcounter{equation}{0}
\begin{equation} \label{eq:int01}
t\sp{1/2}=\frac{1}{\pi}\int\sb{0}\sp{\infty}\frac{t}{\lambda+t}
 \lambda\sp{-1/2}d\lambda.
\end{equation}
Putting  $d\mu(\lambda)=\frac{1}{\pi}\lambda\sp{-1/2}d\lambda$
in (\ref{eq:int01}), we obtain a positive $\sigma$-finite Borel
measure $\mu$ on $(0,\infty)\subseteq\reals$, and we may write
\begin{equation} \label{eq:int02}
t\sp{1/2}=\int\sb{0}\sp{\infty}\frac{t}{\lambda+t}\,d\mu(\lambda)
 \text{\ for\ }0<t\in\reals.
\end{equation}
Now let $a$ be an invertible element in $A\sp{+}$, let $\omega
\in S(A)$, choose a C-block $B$ with $a\in B$ and let $X$
be the Stone space of $P\cap B$. By Lemma \ref{lm:C-block} (i),
there is a synaptic subalgebra $F$ of $C(X,\reals)$ and there is
a synaptic isomorphism $\Psi$ of $B$ onto $F$. As in Lemma \ref
{lm:C-block} (iv), we put $f:=\Psi(a)$, so that $f\sp{1/2}=\Psi
(a\sp{1/2})$. Moreover, as $a$ is invertible, so is $a\sp{1/2}$,
hence also $f\sp{1/2}$, and we have $0<f\sp{1/2}(x)$ for all $x
\in X$. Thus, by Lemma \ref{lm:C-block} (iii) with $f$ replaced
by $f\sp{1/2}$ and (\ref{eq:int02}),
 \begin{equation} \label{eq:int03}
\omega(a\sp{1/2})=\int\sb{X}f\sp{1/2}(x)\, dm\sb{\omega}(x)=\int
 \sb{X}\left(\int\sb{0}\sp{\infty}\frac{f(x)}{\lambda+f(x)}\,
 d\mu(\lambda)\right)dm\sb{\omega}(x).
\end{equation}
Applying Fubini's theorem \cite[Theorem C, p. 148]{Halmos} to
(\ref{eq:int03}), we have
\begin{equation} \label{eq:int04}
\omega(a\sp{1/2})=\int\sb{0}\sp{\infty}\left(\int\sb{X}\frac{f(x)}
 {\lambda+f(x)}\, dm\sb{\omega}(x)\right)\,d\mu(\lambda),
\end{equation}
and combining (\ref{eq:int04}) with Lemma \ref{lm:C-block} (iv),
we obtain the desired integral formula for $\omega(a\sp{1/2})$.
\end{proof}

\begin{theorem}\label{th:MSR}
The synaptic algebra $A$ has the MSR property.
\end{theorem}

\begin{proof}
Suppose that $a,b\in A\sp{+}$, both $a$ and $b$ are invertible,
and $a\leq b$. By Theorem \ref{th:invertibleMSRcase}, it will
be sufficient to prove that $a\sp{1/2}\leq b\sp{1/2}$. Let
$0\leq\lambda\in\reals$. Then $a\leq \lambda+a$, $b\leq \lambda
+b$, and therefore both $\lambda+a$ and $\lambda+b$ are invertible
with $\lambda+a\leq\lambda+b$. Consequently, $(\lambda+b)\sp{-1}
\leq(\lambda+a)\sp{-1}$ by Lemma \ref{lm:inverse} (iii), whence
$\lambda(\lambda+b)\sp{-1}\leq\lambda(\lambda+a)\sp{-1}$, and
therefore $1-\lambda(\lambda+a)\sp{-1}\leq1-\lambda(\lambda+b)
\sp{-1}$. But
\[
1-\lambda(\lambda+a)\sp{-1}=(\lambda+a)(\lambda+a)\sp{-1}-
\lambda(\lambda+a)\sp{-1}
\]
\[
=(\lambda+a-\lambda)(\lambda+a)\sp{-1}=a(\lambda+a)\sp{-1},
\]
likewise $1-\lambda(\lambda+b)\sp{-1}=b(\lambda+b)\sp{-1}$, and
we have $a(\lambda+a)\sp{-1}\leq b(\lambda+b)\sp{-1}$.  Thus,
$\omega(a(\lambda+a)\sp{-1})\leq\omega(b(\lambda+b)\sp{-1})$ for
all $\omega\in S(A)$, and by Theorem \ref{th:omegaSRintformula},
we infer that $\omega(a\sp{1/2})\leq\omega(b\sp{1/2})$. Thus
$a\sp{1/2}\leq b\sp{1/2}$ by Theorem \ref{th:FnlProps} (i).
\end{proof}

\noindent{\bf Concluding Remarks.} In \cite[Exercise V.1.10]{Bh}, R.
Bhatia outlines a proof that if $r\in\reals$ with $0<r<1$, then
the function $f(x):=x\sp{r}$ for $0<x\in\reals$ is operator monotone
on Hermitian matrices. Our proof with $r=1/2$ for a synaptic algebra
is partially based on Bhatia's argument. Pioneering work on monotone
functions of matrices and operators was conducted by K. L\"{o}wner in
\cite{KarlL} and G. Pedersen gave a short proof of L\"{o}wner's
operator monotone theorem in \cite{GertP}. It would be interesting
and significant to determine the extent to which the results of
L\"{o}wner, Bhatia, Pedersen, et al. can be extended to synaptic
algebras, and we hope that our work in this paper might provide a
point of departure for such a project.


\begin{thebibliography}{WW}

\bibitem{Alf} Alfsen, E.M., \emph{Compact Convex Sets and Boundary Integrals}, Springer-Verlag, Heidelberg, New York, 1971.

\bibitem{Bh} Bhatia, Rajendra,  \emph{Matrix Analysis}, Springer-Verlag, New York
Inc., 1997.

\bibitem{FSA} Foulis, D.J., Synaptic algebras, Math. Slovaca {\bf 60} (2010)
631--654.

\bibitem{FPproj} Foulis, D.J. and Pulmannov\'a, Sylvia, Projections in a synaptic
algebra, \emph{Order} {\bf 27} (2010) 235--257.

\bibitem{FPRegGHA} Foulis, D.J. and Pulmannov\'{a}, Sylvia, Regular elements in
generalized Hermitian algebras, \emph{Math. Slovaca} {\bf 61}, no. 2 (2011)
155--172.

\bibitem{TDSyn} Foulis, D.J. and Pulmannov\'{a}, Sylvia, Type-decomposition
of a synaptic algebra, \emph{Found. Phys.} {\bf 43}, no. 8 (2013) 948--968.

\bibitem{SymSyn} Foulis, D.J. and Pulmannov\'{a}, Sylvia, Symmetries in
synaptic algebras, \emph{Math. Slovaca} {\bf 64}, no. 3 (2014) 751--776.

\bibitem{ComSyn} Foulis, D.J. and Pulmannov\'{a}, Sylvia, Commutativity in
a synaptic algebra, to appear in \emph{Math. Slovaca}.

\bibitem{2PSyn} Foulis, David J., Jen\v{c}ov\'{a}, Anna, and Pulmannov\'{a}, Sylvia,
Two projections in a synaptic algebra, \emph{Linear Algebra Appl.} {\bf 478}
(2015) 162-–187.

\bibitem{P&ESyn} Foulis, David J., Jen\v{c}ov\'{a}, Anna, and Pulmannov\'{a}, Sylvia,
A projection and an effect in a synaptic algebra, \emph{Linear Algebra
Appl.}, {\bf 485} (2015) 417--441.

\bibitem{VLSyn} Foulis, David J., Jen\v{c}ov\'{a}, Anna, and Pulmannov\'{a}, Sylvia,
Vector lattices in synaptic algebras, in preparation.

\bibitem{S&OSA} Foulis, David J., Jen\v{c}ov\'{a}, Anna, and Pulmannov\'{a}, Sylvia,
States and observables on synaptic algebras, in preparation.

\bibitem{Good} Goodearl, K.R., \emph{Partially Ordered Abelian Groups with
Interpolation}, A.M.S. Math. Surveys and Monographs, No. 20, 1980.

\bibitem{Halmos} Halmos, P.R., \emph{Measure Theory}, Springer-Verlag, New York
Inc., 1970.

\bibitem{KadRing} Kadison, R.V. and Ringrose, J.R.. \emph{Fundamentals of the Theory of
Operator Algebras I}, Academic Press Inc., New York, 1983.

\bibitem{KarlL} L\"{o}wner, Karl, \"{U}ber monotone Matrixfunktionen, (German)
\emph{Math. Z.} {\bf 38}, no. 1 (1934) 177-–216.

\bibitem{GertP} Pedersen, Gert, Some operator monotone functions, \emph{Proc.
A.M.S.} {\bf 16}, no. 1 (1972) 309--310.

\bibitem{PuNote} Pulmannov\'{a}, Sylvia, A note on ideals in synaptic algebras,
\emph{Math. Slovaca} {\bf 62}, no. 6 (2012) 1091-–1104.

\bibitem{TopVL} Topping, David M., Vector lattices of self-adjoint operators,
\emph{Trans. Amer. Math. Soc.} {\bf 115} (1965) 14-–30.

\end{thebibliography}
\end{document}